\documentclass[11pt]{article}

\usepackage[english]{babel}
\usepackage{tabularx} 
\usepackage[pdftex]{graphicx}
\usepackage[utf8]{inputenc}
\usepackage{amsmath,amssymb,amsfonts} 
\usepackage{amsthm} 
\usepackage{color} 
\usepackage[round,longnamesfirst]{natbib}
\usepackage{graphicx}
\usepackage[hyperfootnotes=true,plainpages=false,pdfpagelabels=true]{hyperref}
\usepackage{geometry}
\usepackage{calc} 
\newtheorem{thm}{Theorem}
\newtheorem{lem}{Lemma}

\newtheorem{definition}{Definition}
\newtheorem{remark}{Remark}\theoremstyle{remark}

\def\R{\mathbb{R}}
\def \esp{\mathbb{E}}
\def \esps{\mathbb{E}^\star}
\def\ynew{Y_{\mbox{\tiny new}}}
\def\cov{{\mbox{cov}}}

\def\bU{\boldsymbol{U}}
\def\bu{\boldsymbol{u}}
\def\bY{\boldsymbol{Y}}
\def\bZ{\boldsymbol{Z}}
\def\bz{\boldsymbol{z}}
\def\bbeta{\boldsymbol{\beta}}
\def\btheta{\boldsymbol{\theta}}
\def\bmu{\boldsymbol{\mu}}
\def\bepsilon{\boldsymbol{\varepsilon}}
\def\hbbeta{\hat{\bbeta}}
\def\hbtheta{\hat{\btheta}}
\def\Tr{{\rm tr}}

\def\vec{{\rm vec}}

\def\FisherInfo{{\mathcal I}}

\def\iid{{\it i.i.d.}~}

\def\opt2m{{I^\star}}

\def\clQ{\mathcal{Q}}

\title{AIC, $C_p$ and estimators of loss for elliptically symmetric distributions} 
\author{Aur\'elie Boisbunon, 
\thanks{
Normandie Université, Universit\'e de Rouen, LITIS EA 4108,
Avenue de l'Universit\'e, BP 12, 76801 Saint-\'Etienne-du-Rouvray,
France.
}
\and        St\'ephane Canu, 
\thanks{
Normandie Université, INSA de Rouen, LITIS EA 4108,
Avenue de l'Universit\'e, BP 8, 76801 Saint-\'Etienne-du-Rouvray,
France.
}
\and  Dominique Fourdrinier
\thanks{
Universit\'e de Rouen, LITIS EA 4108,
Avenue de l'Universit\'e, BP 12, 76801 Saint-\'Etienne-du-Rouvray,
France,
and Cornell University, Department of Statistical
Science,
1176 Comstock Hall, Ithaca, NY 14853, USA
(e-mail: dominique.fourdrinier@univ-rouen.fr).
}
\and William Strawderman
\thanks{
Rutgers University, Department of Statistics
(e-mail: straw@stat.rutgers.edu).
}
\and Martin T. Wells
\thanks{
Cornell University, Department of Statistical Science,
1190 Comstock Hall, Ithaca, NY 14853, USA.
(e-mail: mtw1@cornell.edu).
}
}
 
\begin{document}

\maketitle{}

\begin{abstract}
In this article, we develop a modern perspective on Akaike's Information Criterion (AIC)
and Mallows' $C_p$ for
model selection, and  proposes  generalizations to spherically {and elliptically} symmetric distributions.
Despite the differences in their respective
motivation, $C_p$ and AIC  are equivalent in the special case of Gaussian linear
regression. In this case they are also equivalent to a third criterion, an unbiased estimator of the
quadratic prediction loss, derived from loss estimation theory.
We then show that the form of the unbiased estimator of the quadratic prediction loss
under a Gaussian assumption still holds under  a more general distributional assumption, the family of spherically  symmetric distributions.
One of the features of our results is that our criterion does not rely on the specificity of the distribution, but only on its spherical symmetry.
The same kind of criterion can be derived for a family of elliptically contoured distribution, which allows correlations, when considering the invariant loss.
More specifically, the unbiasedness property is relative to a distribution associated to the original density.
\end{abstract}


%
%

\section{Introduction}
\label{sec:intro}

The problem of model selection has generated a lot of interest
for many decades now and especially recently
with the increased  size of  datasets.
  {In such a context, it is important to model}
the  observed data in a way that accounts for the sparsity of the parameter space.
The principle of parsimony helps to avoid classical issues such as overfitting or computational 
error. At the same time, the model should capture sufficient
information in order to comply with some objectives of good
prediction, good estimation or good selection and thus it should not
be too sparse. This principle has been elucidated by many statisticians as a
trade-off between goodness of fit to data and complexity of the model
\citep*[see for instance][Chapter 7]{hastie2008elements}.
From the practitioner's point of view, model selection is often implemented through
cross-validation \citep[see][for a review on this topic]{arlot2010survey}
or the minimization of criteria whose theoretical justification relies on hypotheses made within a given framework.
In this paper,
we review two of the most commonly used criteria,
namely Mallows' $C_p$  and Akaike's AIC, together with the associated theory
 under Gaussian distributional assumptions,
and then we propose a generalization to spherically {and elliptically} symmetric distributions.

We will focus initially on the linear regression model
\begin{equation}
  \label{eq:LM1}
  Y = X\beta+\sigma\varepsilon,
\end{equation}
where $Y$ is a random vector in $\mathbb{R}^n$, $X$ is a fixed and
known full rank design matrix containing $p$ observed variables $\mathbf{x}^j$ in
$\mathbb{R}^{n}$, $\beta$ is the unknown vector in $\mathbb{R}^p$ of
regression coefficients to be estimated,
{$\sigma$ is the  noise level} and $\varepsilon$ is a random vector in
$\mathbb{R}^{n}$ representing the model noise, with mean zero and
covariance matrix proportional to the identity matrix
(we assume the proportion coefficient to be equal to one when $\varepsilon$  is  Gaussian).
One subproblem of model selection is the problem of variable
selection: only a subset 
of the variables in $X$ gives sufficient
and nonredundant information on $Y$ and we wish to recover this subset
as well as correctly estimate the corresponding regression coefficients.

Early work treated the model selection problem from the hypothesis testing point
of view. For instance the  Forward Selection and Backward
Elimination procedures were stopped using appropriate critical values.
 This practice  changed with Mallows' automated criterion
known as $C_p$ \citep{mallows1973some}. Mallows' idea was to propose an
unbiased estimator of the scaled expected prediction error
$\mathbb{E}_\beta[\|X\hat\beta_I-X\beta\|^2/\sigma^2]$,
where $\hat\beta_I$ is an estimator of $\beta$
based on the selected variables set $I \subset\{1,\dots,p\}$,
$\mathbb{E}_\beta$ denotes the expectation with respect to the sampling distribution in model \eqref{eq:LM1}
and $\|\cdot\|$ is the Euclidean norm on $\mathbb{R}^n$. Assuming Gaussian \iid error terms, Mallows proposed the following
criterion
\begin{equation}\label{eq:Cp}
  C_p = \frac{\|Y - X\hat\beta_I \|^2\!\!\!}{\hat\sigma^2}\;+2\widehat{df}-n,
\end{equation}
where $\hat\sigma^2$ is 
an estimator of the variance $\sigma^2$ based on the full linear model fitted with
the least-squares estimator $\hat\beta^{LS}$, that is,
$\hat\sigma^2=\| Y - X\hat\beta^{LS} \|^2/(n-p)$, and
$\widehat{df}$ is an estimator of $df$, the
degrees of freedom, also called the effective dimension of
the model \citep[see][]{hastie1990generalized,meyer2000degrees}.
For the least squares estimator, {$df$ is the number $k$ of variables in the selected subset~$I$}. 
Mallows' $C_p$ relies on the assumption that, if for some subset $I$
of explanatory variables
the expected prediction error is low, then those
variables to be relevant for predicting $Y.$ In practice, the rule for
selecting the ``best'' candidate is the minimization of $C_p.$
However, Mallows argues that this rule should not be applied in all
cases, and that it is better to look at the shape of the $C_p$-plot instead,
especially when some explanatory variables are highly correlated.

In 1974, Akaike proposed different
automatic criteria that would not need a subjective calibration of the
significance level as in hypothesis testing based approaches. His proposal was
more general with application to many problems such as variable selection, factor analysis,
analysis of variance, or order selection in autoregressive models
\citep{akaike1974new}. Akaike's motivation however was 
different from Mallows.
Akaike considered  the problem of estimating the density $f(\cdot|\beta)$ of an outcome variable
$Y$, where $f$ is parameterized by $\beta\in\mathbb{R}^p$. 
His aim was to
generalize the principle of maximum likelihood enabling a selection
between several maximum likelihood estimators $\hat\beta_I$.
Akaike
showed that all the information for discriminating the estimator $f(\cdot|\hat\beta_I)$ from the true $f(\cdot|\beta)$ could be summed up by the
Kullback-Leibler divergence
$D_{KL}(\hat\beta_I,\beta) =\mathbb{E}[\log
f(\ynew|\beta)]-\mathbb{E}[\log f(\ynew|\hat\beta_I)]$
where the expectation is taken over new observations.
An accurate quadratic approximation to this divergence is possible when
$\hat\beta_I$ is sufficiently close to $\beta$, which actually
corresponds to the distance $\|\hat\beta_I-\beta\|^2_{\FisherInfo}/2$ where
$\FisherInfo=-\mathbb{E}[(\partial^2 \log  f/\partial
    \beta_i\partial \beta_j)_{i,j=1}^p]$ is the
Fisher-information matrix and for a vector $\mathbf{z}$, its weighted norm $\|\mathbf{z}\|_{\FisherInfo}$ is defined by
$(\mathbf{z}^t \FisherInfo \mathbf{z})^{1/2}$. By means of asymptotic analysis
and by considering the expectation of $D_{KL}$ 
Akaike arrived at the following criterion
\begin{equation}
  \label{eq:AIC}
  \text{AIC} = -2\sum_{i=1}^n\log f(y_i|\hat\beta_I)+2k,
\end{equation}
where $k$ is the number of parameters of $\hat\beta_I$. In the special case of
a Gaussian distribution, AIC and $C_p$ are equivalent up to a
constant for model \eqref{eq:LM1} (see Section \ref{sec:link}). Hence Akaike described his AIC as a generalization of $C_p$ to a more general class of models. 
Unlike Mallows, Akaike explicitly recommends the rule of minimization
of AIC 
to identify the best model from
data. Note that \cite{ye1998measuring} proposed to extend AIC to more complex settings by replacing $k$
by the estimated  degrees of freedom $\widehat{df}$.

\enlargethispage{.2cm}

Both {$C_p$ and AIC} have been criticized in the
literature, especially for the presence of the constant 2 tuning the adequacy complexity/trade-off and favoring complex
models in many situations, and many authors have proposed some
correction  \citep[see][]{schwarz1978estimating,foster1994risk,shibata1980asymptotically}.
Despite these critics, these criteria are still quite popular among
practicioners. Also they can be very useful in deriving better criteria
of the form $\delta=\delta_0-\gamma$, where $\delta_0$ is equal
to $C_p$, AIC or an equivalent and $\gamma$ is a correction function
based on data. This framework, referred to as loss estimation, has
been successfully used by
\cite{johnstone1988inadmissibility} and \cite{fourdrinier1995estimation}, among
others, to propose good criteria for selecting the best submodel.

Another possible criticism of $C_p$ and AIC regards their strong
distributional assumptions. Indeed, $C_p$'s unbiasedness in the \iid
Gaussian  case, while AIC requires  specification of the  distribution.
However, in many practical cases, we might
not have any prior knowledge or intuition about the form of the
distribution, and we want the result to be robust with respect to a wide family of
distributions.

The purpose of the present paper is threefold:
\begin{itemize}
\item First, we show in Section \ref{sec:loss-estim} that the
  procedures $C_p$ and AIC are equivalent to unbiased estimators of
  the quadratic prediction loss when $Y$ is assumed to be
  Gaussian in model  \eqref{eq:LM1}. This result is an   important initial step
   for deriving improved criteria as
  is done in
  \cite{johnstone1988inadmissibility} and \cite{fourdrinier2012improved}.
  Both references consider the case of improving the
  unbiased estimator of loss based on the data.
The derivation of  better criteria will not be covered in the present article.
\item Secondly,
  we derive the unbiased loss estimator for the wide family of
  spherically symmetric distributions and show that, for any spherical
  law, this unbiased estimator is the same as that derived under the
  Gaussian model.
 The family of spherically symmetric distribution is a large family which generalizes the multivariate standard normal  law.
  Also, the spherical assumption frees us from the independence assumption of the error terms in  \eqref{eq:LM1},
  while not rejecting it since the Gaussian law is itself spherical.
  Furthermore, some members of the spherical family, like the Student law, have heavier tails than the Gaussian density
  allowing a better handling of potential outliers.
  Finally, the results of the
  present work do not depend on the specific form of the distribution.
  The last two points provide some distributional robustness.
\item {Thirdly, {we extend these results to a model with possible correlation structure. Indeed,} for a family of elliptically symmetric distributions with unknown scale matrix $\Sigma$,  we derive an unbiased estimator of the invariant loss associated to $\Sigma$.
To this end, the notion of unbiasedness has to be adapted  to take into account the fact that $\Sigma$ is unknown.
Our results cover a multivariate  version of the regression model where it is possible to estimate the unknown covariance matrix from the observed data.}
\end{itemize}


\section{Expression of AIC and $C_p$ in the loss estimation framework}
\label{sec:loss-estim}

\subsection{Basics of loss estimation}
\label{sec:loss-theory}

The idea underlying the estimation of loss is closely related to Stein's
Unbiased Risk Estimate \citep[SURE,][]{stein1981estimation}. The theory
of loss
estimation was initially developed for problems of estimation of the
location parameter of a multivariate distribution \citep[see {\em e.g.}][]{johnstone1988inadmissibility,fourdrinier2003bayes}.
The principle is classical in statistics and goes as follow: we wish to evaluate the accuracy of a
decision rule
$\hat\mu$ for estimating the unknown location parameter $\mu$ (in the
linear model \eqref{eq:LM1}, we have $\mu=X\beta$). Therefore we define a
loss function, which we write $L(\hat\mu,\mu)$, measuring the
discrepancy between $\hat\mu$ and $\mu$.
Typical examples are the quadratic loss
$ \|\hat\mu-\mu\|^2$ and the invariant (quadratic) loss
$ \|\hat\mu-\mu\|^2/\sigma^2$.
Since $L(\hat\mu,\mu)$ depends on unknown parameters $\mu$ or $(\mu,\sigma^2)$, it is
unknown as well and can thus be assessed through an estimation using
the observations (see for instance \citealp{fourdrinier2012improved},
and references therein for more details on loss estimation).

 {In this article, we only consider  unbiasedness as the notion of optimality. In such a case, unbiased estimators of the loss and unbiased estimators of the corresponding risk, defined by
\begin{equation}\label{eq:risk_def}
R(\hat\mu,\mu) := \mathbb{E}_{\mu}[L(\hat\mu,\mu)],
\end{equation}
are actually defined in the same way
\citep[see][]{stein1981estimation,johnstone1988inadmissibility}. In the sequel, we choose to refer to them as unbiased loss estimators.}

\begin{definition}[Unbiasedness] \label{esb}
Let $Y$ be a random vector in $\mathbb{R}^n$ with mean
$\mu\in\mathbb{R}^n$ and let $\hat\mu(Y)$ be any estimator of $\mu$. An estimator $\delta_0(Y)$ of the loss $L(\hat\mu(Y),\mu)$
is said to be
\emph{unbiased} if, for all $\mu\in\mathbb{R}^n$,  
\begin{displaymath}
  \mathbb{E}_{\mu}\bigl[\delta_0(Y)\bigr] = \mathbb{E}_{\mu}\bigl[L\bigl(\hat\mu(Y),\mu\bigr)\bigr], 
\end{displaymath}
where $\mathbb{E}_{\mu}$ denotes the expectation with respect to the  distribution of $Y$. 
\end{definition}
Obtaining an unbiased estimator of the quadratic loss
requires Stein's identity \citep[see][]{stein1981estimation} which states that,
if $Y \sim \mathcal{N}_n(\mu,\sigma^2I_n)$ and $g:\mathbb{R}^n \rightarrow \mathbb{R}^n$ is a weakly differentiable
  function,
then,       assuming the  expectations exist,
     \begin{equation}
    \label{eq:Stein-id}
    \mathbb{E}_\mu\left[(Y-\mu)^tg(Y)\right] =
    \sigma^2 \,  \mathbb{E}_\mu\left[{\rm div}_Y    g(Y)\right]  ,
  \end{equation}
  where ${\rm div}_Y g(Y) = \sum_{i=1}^n \partial g_i(Y) / \partial
  Y_i$ is the weak divergence of $g(Y)$ with respect to $Y$.
See {\em e.g.}  Section 2.3 in
  \cite{fourdrinier1995estimation} for the definition and the justification of weak differentiability.

 {When dealing with the invariant loss $\|X\hat{\beta}-X\beta\|^2/\sigma$, the need for estimating $\sigma^2$ leads to the following Stein-Haff identity
 \citep[Lemma 3.1 in][]{fourdrinier2012improved}:
     \begin{equation}
    \label{eq:Stein-id_inv}
    \mathbb{E}_{\mu,\sigma^2}\left[\frac{1}{\sigma^2} \; h(Y,S)\right] =
     \mathbb{E}_{\mu,\sigma^2} \! \left[ (n-p-2) S^{-1} h(Y,S) \; + \; 2 \frac{\partial}{\partial S} h(Y,S)      \right]  ,
  \end{equation}
where $S$  is a non negative random variable independent of $Y$ such that $S \sim \sigma^2 \chi_{n-p}^2$
and $h:\mathbb{R}^n \times \mathbb{R}^+ \rightarrow \mathbb{R}^n$ is a weakly differentiable function
with respect to $S$.  Note that a slightly different notion of unbiasedness will be used in the elliptical case.}

\subsection{Unbiased loss estimation for the linear regression}

 When considering the
Gaussian model in \eqref{eq:LM1},
we have $\mu=X\beta$, we set  $\hat\mu=X\hat\beta$ and $L(\hat\beta,\beta)$ is defined as the quadratic
loss $\|X\hat\beta-X\beta\|^2$.
Special focus will be given to the quadratic loss since it
is the most commonly used and allows tractable calculations.  In practice, it is a
reasonable choice if we are interested in both good selection and
good prediction at the same time. Moreover,  quadratic loss allows
us to link loss estimation with $C_p$ and AIC.

 In the following theorem,
 an unbiased estimator of the quadratic loss, under a Gaussian assumption,  is developed using a result of  \cite{stein1981estimation}.
\begin{thm}
\label{th_gauss}
  Let $Y \sim \mathcal{N}_n(X\beta,\sigma^2I_n)$.  
  Let $\hat\beta = \hat\beta(Y)$ be a function of the least squares estimator of $\beta$ such that $X\hat\beta$  is weakly
  differentiable with respect to $Y$.
  Let
   $\hat\sigma^2   = \|Y - X\hat\beta^{LS} \|^2/(n-p)$.
  Then
 \begin{equation}
   \label{eq:delta0}
    \delta_0(Y) = \|Y - X\hat\beta\|^2 + (2 \, {\rm div}_{Y }(X\hat\beta) - n)\hat\sigma^2
  \end{equation}
 is the usual  unbiased estimator of $\|X\hat\beta - X\beta \|^2$.

\end{thm}
%
%
\begin{proof}
The  risk of  $X\hat\beta$ at $X\beta$ is
  \begin{eqnarray}
    \label{eq:risk}
    \!\!\!\!  \mathbb{E}_\beta[\|X\hat\beta - X\beta\|^2]  
    &=& \mathbb{E}_\beta[\|X\hat\beta -
    Y\|^2 +\| Y- X\beta \|^2]\\
    &&\qquad \qquad  +\mathbb{E}_\beta[ 2(Y-X\beta)^t(X\hat\beta-Y)]  . \nonumber
  \end{eqnarray}
Since $Y\sim
  \mathcal{N}_n(X \beta,\sigma^2I_n)$, we have
{  $\mathbb{E}_\beta[\| Y - X \beta \|^2] =  \mathbb{E}_\beta[(Y - X \beta)^tY] = n \, \sigma^2$
  leading to
  \begin{equation} \nonumber
    \mathbb{E}_\beta[\|X\hat\beta - X\beta\|^2]   \!=\!   \mathbb{E}_\beta[\|Y - X\hat\beta\|^2]   - n \, \sigma^2  +  2 \, \Tr({\rm cov}_\beta(X\hat\beta,Y-X\beta)).
  \end{equation}}
  Moreover, applying Stein's identity for the right-most part of the
  expectation in \eqref{eq:risk} with $g(Y) = X\hat\beta$ and assuming that $X\hat\beta$ is weakly
differentiable with respect to $Y$, we can rewrite \eqref{eq:risk} as
\begin{displaymath}
  \mathbb{E}_\beta[\|X\hat\beta - X\beta \|^2]  = \mathbb{E}_\beta[\|Y - X\hat\beta\|^2] -n \, \sigma^2 + 2\, \sigma^2 \, \mathbb{E}_\beta\!\left[{\rm div}_Y X\hat\beta\right]\!.
\end{displaymath}

{Since 
$\hat\sigma^2$ is an unbiased estimator of $\sigma^2$ 
independent of $\hat\beta^{LS}$, 
the right-hand side of this last equality is
also equal to the expectation of $\delta_0(Y)$ given by
Equation \eqref{eq:delta0}. Hence, according to Definition \ref{esb}, the statistic $\delta_0(Y)$ is  an unbiased estimator of $\|X\hat\beta - X\beta \|^2$.}
\end{proof}

 {\begin{remark}[Estimating the variance]
  The estimation of the variance is a crucial problem in model selection, as different variance estimators can lead to very different models selected. However, in this article, we do not intend to develop new model selection criteria, but only to review and extend existing ones. In particular, the restriction to unbiasedness makes the estimator $\hat{\sigma}^2 = \|Y- X\hat{\beta}^{LS}\|^2/(n-p)$ natural. For more details on the issue, we refer the interested reader to \cite{cherkassky2003comparison}.
\end{remark}}

\begin{remark}
It is often of interest to use robust estimators of $\beta$ and $\sigma^2$.
In such a case, the hypotheses of the theorem need to be modified to insure the independence between estimators $\hat\beta$ and $\hat\sigma^2$
which were implicit in the statement of the theorem.
We will see in Remark \ref{rem:df} of the next sections that, by the use of Stein identities for the spherical and elliptical cases, the implicit  assumption of independence is actually not needed.
\end{remark}

\subsection{ {Unbiased estimation of the invariant loss for  regression}}

 {When dealing with the invariant loss, it is natural to consider estimators of $\beta$ involving the sample variance  $S = \|Y - X\hat\beta^{LS} \|^2$, that is,
 $\hat\beta(Y,S)$.
However,
by consistency with the elliptical case tackled in Section 4, we consider estimators 
that only depend on the least squares estimator and no longer on $S$.
Therefore, we give the following adaptation of Theorem 3.1 from \cite{fourdrinier2012improved}.

\begin{thm}
\label{th_gauss_inv}
  Let $Y \sim \mathcal{N}_n(X\beta,\sigma^2I_n)$ and $n \geq 5$.  
  Let $\hat\beta = \hat\beta(Y)$ be an estimator of $\beta$  weakly differentiable with respect to $Y$ and independent of
 $\|Y - X\hat\beta^{LS} \|^2$.
Then
 \begin{equation}
   \label{eq:delta0_inv}
    \delta_0^{\mbox{\scriptsize inv}}(Y) = \frac{n-p-2}{ \|Y - X\hat\beta^{LS} \|^2}\|Y - X\hat\beta\|^2 + 2 \, {\rm div}_{Y }(X\hat\beta) - n
  \end{equation}
 is an  unbiased estimator of the invariant loss $\|X\hat\beta - X\beta \|^2/\sigma^2$.
\end{thm}

Note that, for more general estimators of the form $\hat\beta(Y,S)$,
a correction term has to be added to    \eqref{eq:delta0_inv}.
Thus, if $\hat\beta(Y,S) = \hat\beta^{LS}(Y) + g(\hat\beta^{LS},S)$ for some function $g$, this correction is $4(X\hat\beta(Y,S) - Y)^tX \frac{\partial g(\hat\beta^{LS},S)}{\partial S} $.}

\subsection{Links between loss estimation and model selection}
\label{sec:link}

In order to make the following discussion clearer, we recall here the
formulas of the three criteria of interest for the
Gaussian assumption.
First the historical criterion Mallows' $C_p$ and the extended version
of AIC proposed by \cite{ye1998measuring}:
$$
C_p(\hat\beta)  =  \frac{\|Y - X\hat\beta\|^2\!\!\!}{\hat\sigma^2} \; + 2 \,   \widehat{df} - n
$$
and
$$
\text{AIC}(\hat\beta)  =
\frac{\|Y - X\hat\beta\|^2\!\!\!}{\hat\sigma^2} \; + 2  \,   {\rm div}_Y (X\hat\beta) .
$$
Secondly  the unbiased estimator of loss $\delta_0(Y)$:
$$
  \delta_0(\hat\beta)  =  \|Y - X\hat\beta\|^2 + (2 \,  {\rm div}_Y (X\hat\beta) - n)\hat\sigma^2.
$$
We have the following link between $\delta_0$, $C_p$ and AIC:
\begin{equation}
  \label{eq:d0_Cp_AIC}
  \delta_0(\hat\beta)={\hat\sigma^2} \times C_p(\hat\beta)={\hat\sigma^2} \times (\text{AIC}(\hat\beta)-n)
\end{equation}
since the statistic ${\rm div}_Y (X\hat\beta)$ is as an unbiased estimator of the (generalized)  degrees of freedom \citep{ye1998measuring}.

These links between different criteria  for model selection are due to the fact that, under our working hypothesis
(linear model, quadratic loss, normal distribution $Y \sim \mathcal{N}_n(X\beta,\sigma^2I_n)$ for a fixed design matrix $X$),
they can be seen as unbiased estimators of related quantities of
interest. 
{Note that there is also an equivalence with other model
  selection criteria, such as those investigated in \cite{li1985stein}, \cite{shao1997asymptotic} 
and \cite{efron2004estimation}. 
}

The final objective is to select the ``best'' model among those at
hand. This can be performed by minimizing either of the three proposed
criteria, that is the unbiased estimator of loss $\delta_0$, $C_p$ and
AIC. The idea behind this  heuristic, as shown in the previous section,
is that the best model in terms
of prediction is the one minimizing the estimated loss. 
Now, from Equation~\eqref{eq:d0_Cp_AIC}, it can be easily seen that the
three criteria
differ from each other only up to a multiplicative and/or additive constant. Hence the models
selected by the three criteria will be the same.

It is important to note that Theorem \ref{th_gauss} does not
{rely on the linearity of the link between $X$ and
  $Y$ so that this work can easily be extended to nonlinear links
.}
 Therefore $\delta_0$  generalizes $C_p$ to nonlinear models.
Moreover,  
 following its definition \eqref{eq:AIC}, AIC implementation requires the specification of  the underlying distribution. 
  In this sense it is considered as a generalization of $C_p$ for
  non Gaussian distributions. However, in practice, we might only have
  a vague intuition of nature of the underlying distribution and we might not be
  able to give its specific form.
We will see in the following section that $\delta_0$, which is
equivalent to the Gaussian AIC as we have just seen, can
  be also derived from a more general distribution context, that of
  spherically symmetric distributions, with no need to specify the
 precise form of the distribution.

 {Note that  $C_p$ and $AIC$ are developed first fixing $\sigma^2$ and then estimating it by  $\hat\sigma^2$,
while, for $\delta_0$, the estimation of   $\sigma^2$ is integrated to the construction process.
It is then natural  to gather the evaluation of $\beta$ and $\sigma^2$ estimating  the invariant loss
$$ \frac{\|X\beta - X\hat\beta \|^2\!\!\!}{\sigma^2} \; , $$
for which $\delta_0^{\mbox{\scriptsize inv}}(\hat\beta)$ in an unbiased estimator.
Note that $\delta_0^{\mbox{\scriptsize inv}}(\hat\beta)$ involves the  variance estimator $ \|Y - X\hat\beta^{LS}\|^2/(n-p-2)$
instead of $ \|Y - X\hat\beta^{LS}\|^2/(n-p)$. This alternative variance estimator was also considered in the unknown variance setting for the construction of the modified $C_p$, which is actually equivalent to $\delta_0^{\mbox{\scriptsize inv}}(\hat\beta)$, and the corrected AIC \citep[see][and references therein]{davies2006estimation}.}

\enlargethispage{.1cm}

\section{Unbiased loss estimators for spherically symmetric distributions}
\label{sec:extension}

\subsection{Multivariate spherical distributions}
\label{sec:spher-sym}

In the previous section, results were given under the Gaussian
assumption with covariance matrix $\sigma^2I_n$. In this section, we
extend this distributional framework. 

\vspace{5pt}
The characterization of the normal distribution as expressed by
\cite{kariya1989robustness} allows two directions for
generalization. Indeed, the authors assert that a random vector is
Gaussian, with covariance matrix proportional to the identity matrix, if and only if its components are independent and its law is
spherically symmetric. Hence, we can generalize the Gaussian
assumption by either keeping the independence property and
consider other laws than the Gaussian, or by relaxing the independence
assumption to the benefit of spherical symmetry. 
In the same spirit,
\cite{fan1985inadmissibility} pointed out that there are two main
generalizations of the Gaussian assumption in the literature: one
generalization comes from the interesting properties of
the exponential form and leads to the exponential family of distributions,
while the other is based on the invariance under orthogonal
transformation and results in the family of spherically symmetric
distributions (which can be generalized by elliptically
contoured distributions). These generalizations are complementary but go in different
directions and have lead to fruitful lines of research. Note that their
only common member is the Gaussian distribution.
The main interest of choosing spherical distributions is that
the conjunction of invariance under orthogonal transformation
together with linear regression with less variables than observations 
brings robustness.
The interest of that property is illustrated by the fact that
some statistical tests designed under a Gaussian assumption, such as the
Student and Fisher tests,
remain valid for spherical distributions
\cite{wang2002power,fang1990statistical}.
This robustness property is not shared by independent non-Gaussian distributions, as mentioned in \cite{kariya1989robustness}.

 {To be self contained, we recall the definition of spherically symmetric distributions.
\begin{definition}
\label{def:SSD}
A random vector $\varepsilon$  is said to be spherically symmetric if,
for any {orthogonal} matrix $Q$,
the law of $Q\varepsilon$ is the same as the one of $\varepsilon$.
\end{definition}
Note that any random vector $\varepsilon$ in Definition  \ref{def:SSD} is necessarily centered at zero.}
Then, for any fixed vector $\mu$, the distribution of $\varepsilon + \mu$ will be said spherically symmetric around $\mu$.

Thus, from the model in \eqref{eq:LM1}, the distribution of $Y$ is the
distribution of $\sigma\,\varepsilon$ translated by $\mu=X\beta$: $Y$ has a spherically symmetric distribution about the location parameter $\mu$ with covariance matrix equal to $n^{-1}\sigma^2\esp[\|\varepsilon\|^2] I_n$ 
where $I_n$ is the identity matrix.
We will write $\varepsilon\sim \mathcal{S}_n(0,I_n)$ and $Y\sim
\mathcal{S}_n(\mu,\sigma^2I_n)$. Examples of this family besides  the Gaussian distribution  are  the Student distribution, the Kotz distribution, or any  variance mixtures of  Gaussian distributions. 

As we will see in the sequel,
the unbiased estimator of the quadratic loss $\delta_0$ \eqref{eq:d0} remains unbiased for any of these
distributions with no need to specify its particular form.
It thus brings distributional robustness.
For more
details and examples of the spherical family, we refer
the interested reader to
{\cite{chmielewski1981elliptically}} for a historical review and
\cite{fang1990generalized} for a comprehensive presentation.

\subsection{The canonical form of the linear regression model}
\label{sec:cf}


An efficient way of dealing with the linear regression model under spherically symmetric distributions
is to use its canonical form  (see \cite{fourdrinier1995estimation} for details).
This form will allow us to give more straightforward proofs.

Considering model \eqref{eq:LM1}, the canonical form consists in an
orthogonal transformation of $Y$.
Using partitioned matrices, let $Q = (Q _1 \; Q_2)$ be an $n \times n$ orthogonal matrix
partitioned such that the first $p$ columns of $Q$
({\em i.e.} the columns of $Q_1$) 
span the column space of $X$.
For instance, this is the case of the $Q$-$R$ factorization
where $X = QR$ with $R$  an $n \times p$ upper triangular matrix.
Now, according to  \eqref{eq:LM1}, let
\begin{equation}  \label{eq:canoZ_spheric}
Q^t Y =
{\begin{pmatrix}
Q_1^t \\
Q_2^t
\end{pmatrix}
}Y =
\begin{pmatrix}
Q_1^t \\
Q_2^t
\end{pmatrix}
X \beta + \sigma Q^t \varepsilon =
\begin{pmatrix}
\theta \\
\mathbf{0}
\end{pmatrix}
 + \sigma Q^t \varepsilon
\end{equation}
with $\theta = Q_1^t X \beta$ and $Q_2^t X \beta = \mathbf{0}$ since the columns of $Q_2$ are orthogonal
to those of $X$. It follows from the definition that $(Z^t \, U^t)^t := Q^t Y$ has a spherically
symmetric distribution about $(\theta^t \, \mathbf{0}^t)^t$. In this sense, the model
\begin{displaymath}
\begin{pmatrix}
Z \\
U
\end{pmatrix}
=
\begin{pmatrix}
\theta \\
\mathbf{0}
\end{pmatrix}
+
\sigma
\begin{pmatrix}
Q_1^t \varepsilon\\
Q_2^t \varepsilon
\end{pmatrix}
\end{displaymath}
is the canonical form of the  linear regression model (\ref{eq:LM1}).

This canonical form has been considered by various authors such as
\cite{CellierFourdrinierRobert1989JMVA,CellierFourdrinier1995JMVA,Maruyama2003SandD, MaruyamaStrawderman2009JMVA,FourdrinierStrawderman2010IMS,
KubokawaSrivastava1999AS}. \cite{KubokawaSrivastava2001JMVA} addressed the multivariate case where $\theta$ is a
mean matrix (in this case $Z$ and $U$ are matrices as well) we will introduce section \ref{sec:cf-elliptic}.

For any estimator $\hat\beta$,  the orthogonality of $Q$ implies that
\begin{equation}\label{eq:decompcannonMC}
\begin{array}{lll}
 \|Y - X\hat\beta\|^2 
           & =&  \| \hat\theta -  Z\|^2 + \| U \|^2
\end{array}
\end{equation}
where  $\hat\theta=Q_1^t X \hat\beta$ is the corresponding estimator
of $\theta$. In particular,
for the least squares estimator $\hat\beta^{LS}$,  we have
\begin{equation}
\|Y - X\hat\beta^{LS} \|^2 =  \| U \|^2 .\label{eq:LS_U}
\end{equation}

In that context,  we recall the Stein-type identity
given by  \cite{fourdrinier1995estimation}.
\begin{thm}[Stein-type identity]\label{stein-type}
  Given $(Z,U)\in\mathbb{R}^n$ a random vector following a spherically
  symmetric distribution around $(\theta,\mathbf{0})$, and
  $g:\mathbb{R}^p \rightarrow \mathbb{R}^p$ a weakly differentiable
  function, we have
  \begin{eqnarray}\label{eq:steinSS}
    \mathbb{E}_\theta[(Z-\theta)^tg(Z)] = \mathbb{E}_\theta[\|U\|^2{\rm div}_Zg(Z)/(n-p)],
  \end{eqnarray}
  provided both expectations exist.
\end{thm}

Note that the divergence in Theorem \ref{th_gauss} is taken with respect to $Y$ while  the Stein type identity  (\ref{eq:steinSS})
requires the divergence with respect to $Z$
(with the assumption of weak differentiability).
Their relationship can be seen in the following lemma.
\begin{lem}\label{lem:divergence}
We have
\begin{equation} \label{eq:traces}
{\rm div}_Y   X \hat\beta(Y) =  {\rm div}_Z   \hat\theta(Z,U) \, .
  \end{equation}
\end{lem}
\begin{proof}
Denoting by ${\rm tr}(A) $ the trace of any matrix $A$  and by
$J_f(x)$ the Jacobian matrix  (when it exists)
of a function $f$ at $x$, we have
\begin{displaymath}
     {\rm div}_Y   X \hat\beta(Y)  =
     {\rm tr}\bigl(J_{X \hat\beta}(Y)\bigl) =  {\rm tr}\bigl(Q^t  \, J_{X \hat\beta}(Y) \, Q\bigl)
  \end{displaymath}
by definition of the divergence and since $Q^t$ is an orthogonal matrix.
  Now, setting $W = Q^t \, Y$, {\em i.e.} $Y = Q \, W$, applying the chain rule to the function
\[
 \widehat{T}(W) = Q^t X \hat\beta(Q \, W)
\]
gives rise to
\begin{equation} \label{eq:traces1}
 J_{ \widehat{T}}(W) = J_{Q^t X \hat\beta}(Y) \, Q = Q^t  \,  J_{X \hat\beta}(Y) \, Q ,
  \end{equation}
noticing that $Q^t$ is a linear transformation.
   Also, as according to (\ref{eq:canoZ_spheric})
\[
 W =
 \begin{pmatrix}
Z \\
U
\end{pmatrix}
 \quad \mbox{and} \quad
 \widehat{T}  = \left(\begin{array}{c} \hat\theta\\\mathbf{0} \end{array}\right) ,
\]
the following decomposition holds
\begin{eqnarray*}
 J_{ \widehat{T}}(W) \! &=& \!
   \begin{pmatrix}
   J_{\hat{\theta}}(Z)  & J_{\hat{\theta}}(U) \\
            \mathbf{0}     &          \mathbf{0}
   \end{pmatrix}   ,
\end{eqnarray*}
where $J_{\hat{\theta}}(Z)$ and $J_{\hat{\theta}}(U)$ are the parts of the Jacobian matrix
in which the derivatives are taken with respect to the components of $Z$ and $U$ respectively.
Thus
\begin{equation} \label{eq:traces2}
{\rm tr}\bigl( J_{ \widehat{T}}(W)\bigr) = {\rm tr}\bigl( J_{ \hat{\theta}}(Z)\bigr)
  \end{equation}
and, therefore, gathering \eqref{eq:traces1}  and  \eqref{eq:traces2},  we obtain
\begin{displaymath}
{\rm tr}\bigl(J_{\hat{\theta}}(Z)\bigr)  =
{\rm tr}\bigl(Q^t \, J_{X \hat\beta}(Y) \, Q\bigr) =
{\rm tr}\bigl(Q Q^t \, J_{X \hat\beta}(Y)\bigr) =
{\rm tr}\bigl(J_{X \hat\beta}(Y)\bigl) ,
  \end{displaymath}
which is  \eqref{eq:traces} by definition of the divergence.
\end{proof}

\subsection{Unbiased estimator of loss for the spherical case}
\label{sec:unbiased-spher}

This section develops a generalization of Theorem
\ref{th_gauss} to the class of spherically symmetric distributions $Y \sim \mathcal{S}_n(X\beta,\sigma^2)$,
given by Theorem \ref{th_spher}.  To do so we need to consider the statistic
\begin{equation}
   \label{eq:VarEstimSansBiais}
\hat\sigma^2(Y)  = \tfrac{1}{n-p}\|Y - X\hat\beta^{LS} \|^2  .
\end{equation}
It is an unbiased estimator of $\sigma^2 \mathbb{E}_\beta  [{\|\varepsilon\|^2}/{n}] $.
Note that, in the normal case where
$Y \sim\mathcal{N}_n(X\beta,\sigma^2I_n)$,
we have $ \mathbb{E}_\beta  [{\|\varepsilon\|^2}/{n}] = 1$ so that $\hat\sigma^2(Y)$ is an unbiased estimator of $\sigma^2$.

\begin{thm}[Unbiased estimator of the quadratic loss under a spherical assumption]\label{th_spher}
   Let $Y \sim\mathcal{S}_n(X\beta,\sigma^2I_n)$ and let $\hat\beta = \hat\beta(Y)$ be
an estimator of $\beta$  depending only on $Q_1^t Y$. If $\hat\beta(Y)$ is
weakly differentiable with respect to $Y$, then the statistic
  \begin{equation}
   \label{eq:d0}
    \delta_0(\hat\beta) =
    \|Y - X\hat\beta(Y)\|^2 + (2 \, {\rm div}_{Y }(X\hat\beta(Y)) - n) \, \hat\sigma^2(Y)
    \, ,
  \end{equation}
is an unbiased estimator of $\|X\hat\beta(Y) - X\beta \|^2$.
 \end{thm}

\begin{proof} 
The  quadratic loss function of  $X\hat\beta$ at $X\beta$ can be decomposed as
  \begin{equation}\label{LossDecomp}
 \|X\hat\beta - X \beta\|^2 =  \|Y - X\hat\beta\|^2 +  \|Y - X \beta\|^2 +
 2 \, (X \hat\beta - Y)^t (Y - X \beta) \, .
  \end{equation}
An unbiased estimator of the second term in the right hand side of (\ref{LossDecomp}) has been
considered in (\ref{eq:VarEstimSansBiais}). As for the third term, by orthogonal invariance of the inner
product,
  \begin{eqnarray}\label{cross-product-term}
   (X  \hat\beta - Y)^t (Y - X \beta) &=& (Q^t X  \hat\beta - Q^t  Y)^t (Q^t  Y - Q^t  X \beta)
   \nonumber \\
   &=&
  \begin{pmatrix}
 Q_1^t X  \hat\beta - Q_1^t Y \\
 Q_2^t X  \hat\beta - Q_2^t Y
  \end{pmatrix}
 ^{t} %
  \begin{pmatrix}
 Q_1^t Y - Q_1^t  X  \beta \\
 Q_2^t Y - Q_2^t X  \beta
  \end{pmatrix} \nonumber \\
    &=&
   \begin{pmatrix}
 \hat\theta - Z \\
 - U
  \end{pmatrix}
 ^{t} %
  \begin{pmatrix}
 Z - \theta \\
 U
  \end{pmatrix} \nonumber \\
  &=&
  (\hat\theta - Z)^t(Z - \theta) - \|U\|^2 \, .
  \end{eqnarray}
Now, since $\hat\theta = \hat\theta(Z,U)$ depends only on $Z$, by Stein type identity, we have
\begin{eqnarray}\label{Stein type identity}
 \esp \! \left[ (\hat\theta - Z)^t(Z - \theta)\right] &=&
 \esp \! \left[ \frac{\|U\|^2}{n-p} \, {\rm div}_Z (\hat\theta - Z) \right]\nonumber \\
 &=&
 \esp \! \left[ \frac{\|U\|^2}{n-p} \left({\rm div}_Z \hat\theta - p\right) \right]
 \end{eqnarray}
so that
  \begin{eqnarray}\label{ExpectInnerProd}
  \esp[(X  \hat\beta - Y)^t (Y - X \beta)] \! \!&=& \! \!
  \esp \! \left[ \frac{\|U\|^2}{n-p} \left({\rm div}_Z \hat\theta - p\right) - \|U\|^2 \right]
  \nonumber \\
  \! \!&=& \! \!
  \esp \! \left[ \frac{\|U\|^2}{n-p} \, {\rm div}_Z \hat\theta - \frac{n}{n-p} \, \|U\|^2 \right]
  \nonumber \\
  \! \!&=& \! \!
  \esp  [\hat\sigma^2(Y) \, {\rm div}_Y X \hat\beta - n \, \hat\sigma^2(Y) ]
  \end{eqnarray}
by (\ref{eq:LS_U}) and since ${\rm div}_Z   \hat\theta = {\rm div}_Y   X \hat\beta$ by Lemma
\ref{lem:divergence}.
Finally, gathering the expressions in (\ref{LossDecomp}), (\ref{eq:VarEstimSansBiais}), and (\ref{ExpectInnerProd})
gives the desired result.
\end{proof}

From the equivalence between $\delta_0$, $C_p$ and AIC under a Gaussian
assumption, and the unbiasedness of $\delta_0$ under the wide class of
spherically symmetric distributions, we conclude that $C_p$ and AIC
derived under the Gaussian distribution is as
 {reasonable} selection criteria for spherically symmetric distributions,
although their original properties may not have been verified in this context.
Also,  { criticisms concerning on the complexity of selected models still stand}.

\begin{remark}\label{rem:df}
  Note that the extension of Stein's lemma in Theorem \ref{stein-type} implies that
  $\widehat{df}={\rm div}_{Y} X \hat\beta$ is also an unbiased estimator of
  $df$ under the spherical assumption.
  Moreover, we would like to point out that the  independence  of  $\hat\sigma^2$  used in the proof of Theorem \ref{th_gauss} in the Gaussian case is no longer necessary.
 Also, to require that  $\hat\beta$  depends on $Q_1^t Y$ only is equivalent to say that  $\hat\beta$ is a function of the least squares estimator.
 When this hypothesis is not available, an extended Stein type identity can be derived \citep{fourdrinier2003robust}.
\end{remark}

\vspace*{-.25cm}

 {\section{Unbiased loss estimators for elliptically symmetric distributions}
\label{sec:elliptic}

Spherical errors clearly generalize the traditional \iid Gaussian case.
However, it is important  to take into account possible correlations among the error components, as in the Gaussian case with a general covariance matrix.
The family of elliptically symmetric distributions is then the natural extension of this Gaussian framework.

Another interesting extension also would have been to non-Gaussian errors. 
Unfortunately, in that case, the only  context in which Stein-type identities are available is the one of exponential family  \citep{hudson1978natural}
and no transition from the linear regression model to a canonical form, as in  \eqref{eq:canoZ_spheric}, is known.
In return,  the framework of elliptically symmetric distributions allows such an approach.
This section is devoted to deriving an unbiased estimator of loss for this case. 

\vspace*{-.25cm}

\subsection{Multivariate elliptical distributions}
\label{sec:elliptic-distrib}

For a random vector $Y$ that has a spherically symmetric distribution, its components are uncorrelated (and independent in the only Gaussian case).
It is important to consider  possible correlations through a general covariance matrix $\Sigma$.
Regarding the linear regression model, the natural extension of spherically symmetric distributions is
the notion of elliptically symmetric distributions
which may be defined through {$\Sigma^{-1}$-orthogonal} transformations (playing the role of orthogonal transformations in the spherical case).

\begin{definition}
A matrix  $Q$ is said to be {$\Sigma^{-1}$-orthogonal}
if it preserves the inner product associated to $ \Sigma^{-1}$,
that is, for all vectors $x$ and $y$,
$$
x^t  \Sigma^{-1} y = (Qx)^t  \Sigma^{-1} Qy ,
$$
or, equivalently, if
\begin{equation}
  \label{eq:OI}
Q^t  \Sigma^{-1} Q=  \Sigma^{-1} .
\end{equation}
\end{definition}

Elliptically symmetric distributions are defined as follows.

\begin{definition}
\label{def:ESD}
Let $\Sigma$ be a positive definite matrix.
A random vector $\varepsilon$ with  covariance matrix proportional to $\Sigma$ is said to be elliptically symmetric if,
for any {$\Sigma^{-1}$-orthogonal} matrix $Q$,
the law of $Q\varepsilon$ is the same as the one of $\varepsilon$. We note $\varepsilon\sim\mathcal{E}_n(0,\Sigma)$.
\end{definition}
Note that any random vector $\varepsilon$ in Definition  \ref{def:ESD} is necessarily centered.
 For any fixed vector $\mu$, the distribution of $\varepsilon + \mu$ will be said elliptically symmetric around $\mu$.
Clearly, a centered Gaussian random vector with covariance matrix $\Sigma$ is elliptically symmetric distributed.
More generally, for any positive matrix  $\Sigma$ and any random vector $W$ with a spherically symmetric distribution,
the random vector $\Sigma^{1/2} W$ is  elliptically symmetric distributed with  covariance matrix~$\Sigma$ (provided $\cov(W) = I$).

Thus, the spherically symmetric  distributions given in subsection 3.1
lead to examples of elliptically symmetric  distributions taking into account possible correlations.

\subsection{The canonical form of the linear regression model in the elliptical case}
\label{sec:cf-elliptic}

In the elliptical case, it is worth considering the multivariate linear regression model, for instance in \cite{KubokawaSrivastava1999AS}
\begin{equation}
  \label{eq:LM2_matrix}
  \bY = X\bbeta+ \bepsilon,
\end{equation}
where $\bY$ is an $n \times m$ matrix of response variable,
$X$ the $n \times p$  known design matrix,
$\bbeta$ is $p \times m$ matrix of unknown parameters
and
$\bepsilon$  an $n \times m$ random matrix of errors
having an elliptically symmetric distribution with scale matrix $\Sigma$.
This approach allows to take into account the fact that $\Sigma$ is unknown.
We assume also that $p + m +1< n $.
In the line of the previous section, we now specify the ellipticity of the distribution of the matrix $\bepsilon$ through its (line based) vectorization $\vec(\bepsilon^t)$ such that $\bepsilon_{ij} = \vec(\bepsilon^t)_{m(i-1)+j}$.
Note that considering column base vectorization would have made trouble for us in the canonical decomposition of the  error  in (\ref{eq:canonique_elliptic}).

In the following, we consider a specific subclass of elliptical distributions for matrices:
the class of vector-elliptical (contoured) matrix distributions \citep[see][]{fang1990generalized}.
Thus, regarding the distribution of $\bepsilon$, we assume that,  
and for any  $(I_n \otimes   \Sigma^{-1})$-orthogonal transformation $nm \times nm$  matrix $\clQ$,
the law of $\clQ \, \vec(\bepsilon^t)$ is the same as that of $\vec(\bepsilon^t)$.
Note that, if $\bepsilon$ has a density, it is of the form
\begin{equation}  \label{eq:densite_elliptical}
\bepsilon \mapsto  |\Sigma|^{-n/2} \, f\bigl(\Tr(\bepsilon \Sigma^{-1} \bepsilon^t  ) \bigr) ,
\end{equation}
for some function $f$.
The importance of the transposition is underlined by \cite{muirhead2009aspects}.
When the covariance exists, it is proportional to the scale matrix $\Sigma$.
 In that case, let $\tau^2$ be this proportionality coefficient so that
 \begin{equation}\label{eq:bruit}
 \esp \bigl[ \Tr(\bepsilon \Sigma^{-1} \bepsilon^t  ) \bigr] = \esp \bigl[\vec\bigl\{ (\bepsilon \Sigma^{-1/2})^t\bigr\}^t \vec\bigl\{(\bepsilon \Sigma^{-1/2})^t\bigr\}\bigr] = nm \tau^2 ,
\end{equation}
 since $\vec\bigl\{ (\bepsilon \Sigma^{-1/2})^t\bigr\}$ has spherically symmetric distribution with proportionality coefficient $\tau^2$.
See  \cite{fang1990generalized} for an approach through characteristic functions and more general classes of elliptical matrix distributions.

The canonical form of   \eqref{eq:LM2_matrix} can be derived as in the spherical case.
Let $Q = (Q _1 \; Q_2)$  be defined as previously in Section~\ref{sec:cf}. 
According to  \eqref{eq:LM2_matrix}, let
\begin{equation}  \label{eq:canoZ}
Q^t \bY =
{\begin{pmatrix}
Q_1^t \\
Q_2^t
\end{pmatrix}
}\bY =
\begin{pmatrix}
Q_1^t \\
Q_2^t
\end{pmatrix}
X \bbeta +  Q^t \bepsilon =
\begin{pmatrix}
\btheta \\
\mathbf{0}
\end{pmatrix}
 +  Q^t \bepsilon
\end{equation}
with matrices $\btheta = Q_1^t X \bbeta$ and $Q_2^t X \bbeta = \mathbf{0}$ since, by construction, the columns of $Q_2$ are orthogonal to those of $X$.
Now, using \eqref{eq:LM2_matrix} and \eqref{eq:densite_elliptical},
$\bY - X\bbeta$ has  density
\begin{eqnarray*}  
\bepsilon &\mapsto  \quad & |\Sigma|^{-n/2} \, f\bigl(\Tr[Q^t \bepsilon \Sigma^{-1} \bepsilon^t Q ] \bigr) \\
&\quad  = &  |\Sigma|^{-n/2} \, f\left(\Tr \left[ \begin{pmatrix}
\bz -\btheta\\
\bu
\end{pmatrix}
 \Sigma^{-1}\begin{pmatrix}
\bz -\btheta\\
\bu
\end{pmatrix}^t \right]
  \right)
\end{eqnarray*}
and hence  $(\bZ^t \, \bU^t)^t := Q^t \bY$ has an elliptically
symmetric distribution about the corresponding column of $(\btheta^t \, \mathbf{0}^t)^t$. In this sense, the model
\begin{displaymath}
\begin{pmatrix}
\bZ \\
\bU
\end{pmatrix}
=
\begin{pmatrix}
\btheta \\
\mathbf{0}
\end{pmatrix}
+
\begin{pmatrix}
Q_1^t \bepsilon\\
Q_2^t \bepsilon
\end{pmatrix}
\end{displaymath}
is the canonical form of the  linear regression model \eqref{eq:canoZ}.
Note that, if it exists, the joint density of $\bZ$ and $ \bU$ is of the form
\begin{eqnarray}\label{eq:density}
 (\bz,\bu) &\mapsto& |\Sigma|^{-n/2} \, f\bigl(\Tr( (\bz - \btheta) \Sigma^{-1} (\bz - \btheta)^t  )  + \Tr(\bu \Sigma^{-1} \bu^t  )\bigr) ,
\end{eqnarray}
for some function $f$.
Extending  the vector case considered in \cite{fourdrinier2003robust}
we denote by $\esp$
 the expectation with respect to the density in \eqref{eq:density} and
$\esp^{\star}$ the expectation with respect to the density
\begin{eqnarray}\label{F-density}
 (\bz,\bu) &\mapsto& \frac{1}{C} \, |\Sigma|^{-n/2} \,
 F \bigl(\Tr( (\bz - \btheta) \Sigma^{-1} (\bz - \btheta)^t  )  + \Tr(\bu \Sigma^{-1} \bu^t  )\bigr)
\, ,
\end{eqnarray}
where
\begin{eqnarray}\label{bigf}
F(t) =\frac{1}{2} \, \int^{t}_{0} f(u) du.
\end{eqnarray}
and where it is assumed that
\begin{eqnarray} \label{equa1.9}
 C = \int_{\R^{mn}}
 |\Sigma|^{-n/2} \,
 F \bigl(\Tr( (\bz - \btheta) \Sigma^{-1} (\bz - \btheta)^t  )  + \Tr(\bu \Sigma^{-1} \bu^t  )\bigr)
 \, d\bz \, d\bu < \infty \, .
\end{eqnarray}
Note that, in the Gaussian case, these two expectations coincide.

We consider the problem of estimating $\bbeta$ with the invariant loss
\begin{equation}\label{eq:inv_loss}
L(\hbbeta,\bbeta) = \Tr\bigl(  (X\hbbeta - X\bbeta)  \Sigma^{-1}  (X\hbbeta - X\bbeta)^t \bigr)  .
\end{equation}
It is worth noting that this loss can be written through the Frobenius norm associated to $\Sigma^{-1}$
that is, for any $n \times m$ matrix $M$,
$$
\| M \|_{ \Sigma^{-1}}^2 = \Tr(  M  \Sigma^{-1}  M^t ) .
$$
As in the spherical case, the nature of $Q$ implies
\begin{equation}
\label{eq:canonique_elliptic}
\|\bY - X\hbbeta\|_{ \Sigma^{-1}}^2  = \|\hbtheta - \bZ\|_{ \Sigma^{-1}}^2  +  \| \bU \|_{ \Sigma^{-1}}^2 ,
\end{equation}
which is the analog of \eqref{eq:decompcannonMC}.

In this context, we can give the following Stein-type identity for the multivariate elliptical distributions.
\begin{thm}[Stein-type identity]\label{thm:stein-type_eliptic}
  Given $(\bZ,\bU)$ a $n \times m$ random matrix following a vector-elliptical  distribution around $(\btheta,\mathbf{0})$ with scale matrix $\Sigma$ and
  $g:\mathbb{R}^{pm} \rightarrow \mathbb{R}^{pm}$ a weakly differentiable
  function, we have
  \begin{eqnarray}\label{eq:steinES}
    \mathbb{E}_\theta \bigl[ \Tr (\bZ - \btheta) \Sigma^{-1} g(\bZ)^t \bigr] = \mathbb{E}_\theta\Bigl[\frac{\; \|\bU\|_{ \Sigma^{-1}}^2}{(n-p)m}  \, {\rm div}_{\bZ} g(\bZ)\Bigr],
  \end{eqnarray}
  provided both expectations exist.
\end{thm}
Taking into account the fact that the scale matrix of the vectorization of  $(\bZ^t \, \bU^t)^t$ is $I_n \otimes \Sigma$,
Theorem 5 can be derived from the spherical version given in \cite{fourdrinier1995estimation}.

To deduce an unbiased type estimator of loss for the elliptical case,
the following adaption from Lemma 1 in \cite{fourdrinier2003robust} is needed.
\begin{thm}[Stein-Haff-type identity] \label{thm:Haff_stein_eliptic}
Under the conditions of Theorem \ref{thm:stein-type_eliptic},
let $S = \bU^t \bU$.
For any $m \times m$ matrix function $T(\bZ,S)$ weakly differentiable with respect to $S$, we have
\begin{equation}
\esp \bigl[ \Tr \bigr(T(\bZ,S)  \Sigma^{-1} \bigl) \bigl] =  C \; \esps \!  \bigl[ 2 D_{1/2}^\star T(\bZ,S)  +  a \; \Tr\bigr(S^{-1} T(\bZ,S)\bigl) \bigl]
\end{equation}
 with $a = n -p -m -1$ and where the operator $D_{1/2}^\star$ is be defined by
 $$
 D_{1/2}^\star T(\bZ,S) = \sum_{i=1}^m \frac{\partial T_{ii}(\bZ,S)}{\partial S_{ii}} + \frac{1}{2}  \sum_{i\neq j} \frac{\partial T_{ij}(\bZ,S)}{\partial S_{ij}}  .
 $$
\end{thm}

\subsection{Unbiased type estimator of loss}
\label{sec:unbiased-elliptic}

In this section we provide an unbiased type estimator $\delta_0^{\mbox{\scriptsize inv}}$ of the invariant loss  $\|X\hat\beta(Y) - X\beta \|_{\Sigma^{-1}}^2/ \tau^2 $
in the sense of the following definition.
\begin{definition}[$\esps$-unbiasedness] \label{esb_star}
Let $\bY$ be a random matrix in $\mathbb{R}^{nm}$ with mean
$\bmu\in\mathbb{R}^{nm}$ and let $\hat\bmu(\bY)$ be any estimator of $\bmu$. An estimator $\delta_0(\bY)$ of the  loss
$L(\hat\bmu(\bY),\bmu)$ 
is said to be
\emph{$\esps$-unbiased} if, for all $\bmu\in\mathbb{R}^{nm}$, 
\begin{displaymath}
   \label{eq:Estar_unbiased}
 \esps\bigr[\delta_0(\bY)\bigl] = \esp\bigr[L(\hat\bmu(\bY),\bmu) \bigl] ,
\end{displaymath}
where $\esp$ is the expectation associated with the original density in  \eqref{eq:density}
while $\esps$ is the expectation  with respect to the modified density  \eqref{F-density}.
\end{definition}
In the usual notion of unbiasedness, $\esps$ is replaced by  $\esp$.
The need of $\esps$ is due to the presence of the matrix $\Sigma$.
\begin{thm}[Unbiased type estimator of loss for the elliptical case]
   Let $\bY \sim\mathcal{E}_n(X\bbeta,\Sigma)$   
   and let $\hat\bbeta = \hat\bbeta(\bY)$ be
an estimator of $\bbeta$  depending only on $Q_1^t \bY$. If $\hat\bbeta(\bY)$ is
weakly differentiable with respect to $\bY$, then the statistic
\begin{equation}
   \label{eq:d0_eiptic}
\delta_0^{\mbox{\scriptsize inv}} =   (n-p-m-1)  \|\bY - X\hbbeta\|_{S^{-1}}^2 + 2  \, {\rm div}_{\bY} X \hbbeta(\bY)  - nm
\end{equation}
is an $\esps$-unbiased estimator of the invariant loss. 
\end{thm}
\begin{proof}
The invariant loss \eqref{eq:inv_loss} can be decomposed as
\begin{equation} \label{eq:eliptic_decomp}
\|X\hbbeta - X\bbeta\|_{ \Sigma^{-1}}^2  = \|\bY - X\hbbeta\|_{ \Sigma^{-1}}^2  - \|\bY - X\bbeta\|_{ \Sigma^{-1}}^2
                                                              + 2 \Tr   (X\hbbeta - X\bbeta)  \Sigma^{-1}  (\bY - X\bbeta)^t   .
 \end{equation}
By orthogonally invariance of the classical inner product, the third term verifies
\begin{equation*}
\Tr   (X\hbbeta - X\bbeta)  \Sigma^{-1}  (\bY - X\bbeta)^t = \Tr   (\hbtheta - \btheta)  \Sigma^{-1}  (\bZ - \btheta)^t .
\end{equation*}
Now, since $\hat\btheta = \hat\btheta(\bZ,\bU)$ depends only on $\bZ$, by the Stein type identity given in Theorem \ref{thm:stein-type_eliptic}, we have
\begin{eqnarray}\label{Stein type identity eliptique}
 \esp \! \left[ \Tr   (\hbtheta - \btheta)  \Sigma^{-1}  (Z - \btheta)^t \right] &=&
 \esp \! \left[ \frac{\|\bU\|_{ \Sigma^{-1}}^2}{(n-p)m} \, {\rm div}_{\bZ}  (\hat\btheta) \right] , \nonumber
  \end{eqnarray}
so that, thanks to  \eqref{eq:bruit}, taking expectation in  \eqref{eq:eliptic_decomp} gives
$$
 \esp \! \left[ \|X\hbbeta - X\bbeta\|_{ \Sigma^{-1}}^2  \right] =  \esp \! \left[ \|\bY - X\hbbeta\|_{ \Sigma^{-1}}^2  \right]  - nm\tau^2
                                                                                          + 2 \,  \esp \! \left[ \frac{\|\bU\|_{ \Sigma^{-1}}^2}{(n-p)m} \, {\rm div}_{\bZ}  (\hat\btheta) \right].
$$

%
%
We get rid of $\Sigma^{-1}$ applying Stein-Haff-type identity (Theorem \ref{thm:Haff_stein_eliptic}) two times.
First     $T(\bZ,S) = (\bY-X\hbbeta)^t (\bY - X\hbbeta)$ being independent of $S$ we have
\begin{equation*}
\esp \! \left[ \|\bY - X\hbbeta\|_{ \Sigma^{-1}}^2  \right]  =
C \; \esps \! \left[  (n-p-m-1) \|\bY - X\hbbeta\|_{ S^{-1}}^2   \right] .
\end{equation*}
Secondly, setting  $T(\bZ,S) =  {\rm div}_{\bZ}  (\hat\btheta)/(n-p)m  \; S$, we have
\begin{equation*}
 \esp \! \left[ \frac{\|\bU\|_{ \Sigma^{-1}}^2}{(n-p)m} \, {\rm div}_{\bZ}  (\hat\btheta) \right] =
 \esp \! \left[ \Tr   (\Sigma^{-1} S) \; \frac{ {\rm div}_{\bZ}  (\hat\btheta)}{(n-p)m} \right] =
  C \; \esps \! \left[  {\rm div}_{\bZ}  (\hat\btheta) \right] .
\end{equation*}
It is worth noticing that 
$
C =  \tau^2
$.
Indeed, thanks to  \eqref{eq:bruit},  $ \esp \! \left[ \|\bU\|_{ \Sigma^{-1}}^2 \right] = (n-p)m \tau^2$
and,
 using Theorem \ref{thm:Haff_stein_eliptic},  $ \esp \! \left[ \|\bU\|_{ \Sigma^{-1}}^2 \right] = C(n-p)m$.
 Finally note that Lemma~\ref{lem:divergence} extends to the matrix case so that 
${\rm div}_{\bZ}   \hat\btheta = {\rm div}_{\bY}   X \hat\bbeta$.
\end{proof}

It is worth noting that  $m=1$ reduces to the spherical case and
that the associated  unbiased estimator of the invariant loss, adapted from equation \eqref{eq:d0_eiptic}, is
\begin{equation} \nonumber
\delta_0^{\mbox{\scriptsize inv}} =   \frac{n-p-2}{\| U \|^2}  \| Y - X\hat\beta\|^2 + 2  \, {\rm div}_{Y} X \hat\beta(Y)  - n ,
\end{equation}
which is the same as the one provided in the Gaussian case in \eqref{eq:delta0_inv}.}

\section{Discussion}
\label{sec:discussion}

In this article we viewed the well-known model selection criteria $C_p$ and AIC
through the lens of loss estimation and related them to an unbiased
estimator of the quadratic prediction loss under a Gaussian
assumption. We then developed unbiased estimators of loss under considerably wider
distributional settings,  spherically symmetric distributions
and the family of elliptically symmetric distributions.
In the spherical context, the unbiased estimator of loss is actually equal to the one derived under the Gaussian law.
Hence, this implies that we do not have to specify the form of the distribution,
the only condition being its spherical symmetry. We also conclude from
the equivalence between unbiased estimators of loss, $C_p$ and AIC
that their form for the Gaussian case is able to handle any
spherically symmetric distribution.
The spherical family is interesting for many practical cases since
it allows a dependence
property between the components of random vectors whenever the
distribution is not Gaussian. Some members of this family also have
heavier tails than Gaussian densities, and thus the unbiased estimator derived here can be robust to outliers.
We also considered a generalization with elliptically symmetric
distributions for the error vector,
 allowing a general covariance matrix $\Sigma$. However, to get such a result,
a matrix regression model which allows estimation of the covariance matrix $\Sigma$  was introduced.

It is well known that unbiased estimators of loss are not the best
estimators and can be improved \citep{johnstone1988inadmissibility,fourdrinier2012improved}.
It was not our intention in this article to derive such estimators, but to explain why their performances can be
similar when departing from the Gaussian assumption.
The improvement over these unbiased estimators requires a way to assess their quality.
This can be done either using oracle inequalities
or the theory of admissible estimators under a certain risk criteria.
Based on a proper risk, a selection rule $\delta_0$ is inadmissible
if we can find a better estimator, say $\delta_\gamma$, that has a smaller risk function for all possible values of the parameter $\beta$,
with strict inequality for some $\beta$. The heuristic of loss estimation is that the closer an estimator is to the true loss, the more we expect
their respective minima to be close.
We are currently working on improved
estimators of loss of the type $\delta_\gamma (Y) =
\delta_0(Y)+\gamma(Y)$, where $\gamma(Y)$ can be thought of
as a data driven penalty.
The selection of such a $\gamma$ term is an important, albeit difficult, research direction.

\section*{Acknowlegment}

{This research was
partially supported by ANR ClasSel grant 08-EMER-002, by the Simons Foundation
grant 209035 and  by NSF Grant DMS-12-08488.}

\end{document}